\documentclass[a4paper,11pt]{article}
\usepackage[T1]{fontenc}
\usepackage[utf8]{inputenc}
\usepackage{amsmath,amsthm,amssymb,color,graphicx,bm,mathtools}
\usepackage{algorithm}
\usepackage{algorithmic}
\usepackage{layaureo}
\usepackage{microtype}
\usepackage{emptypage}
\usepackage{indentfirst}
\usepackage{CJKutf8}
\usepackage{enumitem}
\usepackage[T1]{fontenc}
\usepackage{lmodern}
\usepackage{stmaryrd}
\usepackage[affil-it]{authblk}
\usepackage[english]{babel}
\usepackage{blindtext}

\usepackage[all, pdf]{xy}
\usepackage{subfig}
\usepackage{enumitem}
\usepackage{color}
\usepackage{graphicx,booktabs}
\usepackage[bookmarks,colorlinks=true,linkcolor=blue]{hyperref} 

\usepackage{ulem}

\newcommand{\numberset}{\mathbb}
\newcommand{\R}{\numberset{R}} 
\renewcommand{\phi}{\varphi} 
\renewcommand{\chi}{\mathcal{X}} 
\newcommand{\Z}{\numberset{Z}} 

\renewcommand{\epsilon}{\varepsilon}

\newtheorem{theorem}{Theorem}
\newtheorem{prop}[theorem]{Proposition}
\newtheorem{definition}[theorem]{Definition}

\newtheorem{remark}[theorem]{Remark}

\newtheorem{example}[theorem]{Example}

{\left\lbrace\begin{aligned}}%
	{\end{aligned}\right.}

\begin{document}
	\title{Energy minimizing maps with prescribed singularities and Gilbert-Steiner optimal networks}
	\author[1]{Sisto Baldo}
	\author[1]{Van Phu Cuong Le}
	\author[2]{Annalisa Massaccesi}
	\author[1]{Giandomenico  Orlandi}
	
	\affil[1]{Dipartimento di Informatica, Universit\`a di Verona, Italy, e-mail: sisto.baldo@univr.it, vanphucuong.le@univr.it, giandomenico.orlandi@univr.it }
	\affil[2]{Dipartimento di Tecnica e Gestione dei Sistemi Industriali, Universit\`a degli Studi di Padova, Italy, e-mail: annalisa.massaccesi@unipd.it}
	\date{\today}
	\maketitle
	\begin{abstract}
		We investigate the relation between energy minimizing maps valued into spheres having topological singularities at given points and optimal networks connecting them (e.g. Steiner trees, Gilbert-Steiner irrigation networks). We show the equivalence of the corresponding variational problems, interpreting in particular the branched optimal transport problem as a homological Plateau problem for rectifiable currents with values in a suitable normed group. This generalizes the pioneering work by Brezis, Coron and Lieb $\cite{BrezisCoronLieb}$. \end{abstract}
	
	
\section{Introduction}
In their celebrated paper $\cite{BrezisCoronLieb}$, Brezis, Coron and Lieb showed, in the context of harmonic maps and liquid crystals theory, the existence of a close relation between sphere-valued harmonic maps having prescribed topological singularities at given points in $\R^3$ and {\it minimal connections} between those points, i.e., optimal mass transportation networks (in the sense of Monge-Kantorovich) having those points as marginals. This relation was further enlightened by Almgren, Browder and Lieb in $\cite{abh}$, who recovered the results in $\cite{BrezisCoronLieb}$ by interpreting the (minimal connection) optimal transportation problem as a suitable Plateau problem for rectifiable currents having the given marginals as prescribed boundary.

Our aim is to consider minimizing configurations for maps valued into manifolds and with prescribed topological singularities when the energy is possibly more general than the Dirichlet energy, and investigate the connection with Plateau problems for currents (or flat chains) with coefficients in suitable groups. The choice of these groups is linked to the topology of the involved target manifolds. 

In this paper we will consider the particular case where the manifold is a product of spheres and the maps have assigned point singularities, and we will show, in Theorem \ref{thm1} below, that energy minimizing configurations are related with Steiner-type optimal networks connecting the given points, i.e., solutions of the Steiner problem or solutions of the Gilbert-Steiner irrigation problem. The investigation of maps with values into product of spheres arises in several physical problems, such as the study of the  structure of minimizers of two-component Ginzburg-Landau functionals, where the reference (ground state) manifold is a torus ($\mathbb{S}^{1}\times \mathbb{S}^{1}$) (see \cite{Stan}), or the case of Dipole-Free $^3$He-A, where the order parameter takes values into $(\mathbb{S}^{2}\times$ SO(3))$/\Z_{2}$, whose covering space is $\mathbb{S}^{2}\times \mathbb{S}^{3}$ (see \cite{Mermin, Christopher}).
In a companion paper in preparation $\cite{CVO}$ we will discuss and state the results which correspond  to more general situations. Let us also stress that the generalization of the results to a broader class of energies (and thus different norms) is not moot, this being the case, for instance, for dislocations in crystals (see \cite{CoGarMas}).

Steiner tree problems and Gilbert-Steiner (single sink) problems can be formulated as follows: given $n$ distinct points $P_{1},\ldots, P_{n}$ in $\R^{d}$, where $d, n \geq 2$, we are looking for an optimal connected transportation network, $L = \cup_{i=1}^{n-1}\lambda_i$, along which the unit masses initially located at  $P_{1},\ldots, P_{n-1}$ are transported to the target point $P_n$ (single sink); here $\lambda_i$ can be seen as the path of the $i^{\rm th}$ mass flowing from $P_{i}$ to $P_{n}$, and the cost of moving a mass $m$ along a segment with length $l$ is proportional to $lm^{\alpha}$, $\alpha\in[0,1]$. Therefore, we are led to consider the problem
$$
\inf \left\{ I_\alpha(L):\,L = \bigcup_{i=1}^{n-1}\lambda_i\text{ with }\lbrace P_{i}, P_{n} \rbrace \subset \lambda_{i}\text{, for every }i=1,\ldots, n-1 \right\}
\leqno{(I)}
$$
where the energy $I_\alpha$ is computed as $I_\alpha(L)=\int_L |\theta(x)|^\alpha d{\mathcal H}^1(x)$, with $\theta(x) = \sum_{i=1}^{n-1} \mathbf{1}_{\lambda_i}(x)$. Let us notice that $\theta$ stands for the mass density along the network. In particular, we consider the range $\alpha\in[0,1]$:
\begin{itemize}
\item when $\alpha=0$ the problem is equivalent to optimize the total length of the graph $L$, as in the Steiner Tree Problem (STP); 
\item when $\alpha=1$ the problem $(I)$ becomes the well-known Monge-Kantorovich problem;
\item and when $0<\alpha<1$ the problem is known as the Gilbert-Steiner problem, or, more generally, as a branched optimal transport problem, due to the fact that the cost is proportional to a concave function $\theta^{\alpha}$, which favours the clustering of  the mass during the transportation, thus giving rise to the branched structures which characterize the solutions (we refer the reader to $\cite{Bernot2009}$ for an overview on the topic). 
\end{itemize}

In the last decade, the communities of Calculus of Variations and Geometric Measure Theory made some efforts to study (Gilbert-)Steiner problems in many aspects, such as existence, regularity, stability  and numerical feasibility (see for example \cite{Xia, PaSt, MaMa2, MaMa, MariaAnonioAndrea, MariaAnonioAndreaPegonProuff, OuSa, BoLeSa, MaOuVe,  BoOrOu,  BoOu, BoOrOu2} and references therein). Among all the significant results, we would like to mention recent works in $\cite{MaMa2, MaMa}$ and $\cite{BoOrOu, BoOrOu2}$, which are closely related to the present paper. To be more precise, in $\cite{MaMa2, MaMa}$ the authors turn the problem $(I)$ into the problem of mass-minimization of integral currents with multiplicities in a suitable group. For the sake of readability we postpone proper definitions about currents to Section \ref{section2}, in this introduction we only recall that a $1$-dimensional integral current with coefficients in a group can be thought as a formal sum of finitely many curves and countably many loops with coefficients in a given normed abelian group. For instance, considering the group $\Z^{n-1}$ and assigning to the boundary datum $P_{1}, P_{2},\ldots, P_{n-1}, P_{n}$ the multiplicities $e_{1},e_{2},\ldots,e_{n-1},-(e_{1}+\ldots+ e_{n-1})$, respectively (where $\lbrace e_{i} \rbrace_{1 \leq i \leq n-1}$ is the basis of $\R^{n-1}$), we recover the standard model in \cite{MaMa2,MaMa}. 

In fact we can interpret the network $L = \bigcup_{i=1}^{n-1}\lambda_i$ as the superposition of  $n-1$ paths $\lambda_i$ connecting $P_{i}$ to $P_{n}$ labelled with multiplicity $e_{i}$. This point of view  requires a density function with values in $\Z^{n-1}$, which corresponds to the so-called $1$-dimensional current with coefficients in the group $\Z^{n-1}$. Furthermore, by equipping $\Z^{n-1}$ with a certain norm (depending on the cost of the problem), we may define the notion of mass of those currents, and problem $(I)$ turns out to be equivalent to the Plateau problem.
$$
\inf \left\{ \mathbb{M}(T):\,\partial T = e_{1}\delta_{P_1}+e_{2}\delta_{P_2}+\ldots+e_{n-1}\delta_{P_{n-1}}-(e_{1}+e_{2}+\ldots+e_{n-1})\delta_{P_n} \right\}
\leqno{(M)}
$$
where $T$ is a 1-dimensional current with coefficients in the group $\Z^{n-1}$ (again, we refer the reader to the Section $\ref{section2}$ for rigorous definitions). For mass minimization, there is the very useful notion of calibration (see section \ref{section3}), that is, a tool to prove minimality when dealing with concrete configurations (see Example $\ref{examplecalib}$). To be precise, a calibration is a sufficient condition for minimality, see Definition \ref{Calibration} and the following remarks.

In $\cite{BoOrOu, BoOrOu2}$, by using $\cite{MaMa2, MaMa}$,  a variational approximation of the problem $(I)$ was provided through Modica-Mortola type energies in the planar case, and through Ginzburg-Landau type energies (see \cite{ABO2}) in  higher dimensional ambient spaces via $\Gamma$-convergence.  The corresponding numerical treatment is also shown there.

Following $\cite{MaMa2, MaMa}$, $\cite{BoOrOu, BoOrOu2}$, and the strategy outlined in $\cite{abh}$ (relating the energy of harmonic maps with prescribed point singularities to the mass of $1$-dimensional classical integral currents) we provide here a connection between an energy functional with its energy comparable with $k$-harmonic map problem with prescribed point singularities and (Gilbert-)Steiner problems $(I)$. More precisely, let $P_{1},\ldots,P_{n-1}, P_{n}$ in $\R^{d}$ be given, and consider the spaces $H_{i}$ defined as the subsets of $W^{1,d-1}_{\rm loc}(\R^{d}; \mathbb{S}^{d-1})$ where the functions are constant outside a neighbourhood of the segment joining $P_i,P_n$ and have distributional Jacobian $\frac{\alpha_{d-1}}{d}( \delta_{P_i}-\delta_{P_n})$, respectively. Here $\alpha_{d-1}$ is the surface area of the unit ball in $\R^{d}$.

Let $\mathbb{\psi}$ be a norm on $\R^{n-1}$ which will be specified in Section $\ref{section3}$ (see $\eqref{normeuclidean}$), and set 
\begin{equation}\label{def_h}
	\mathbb{H}({\bf u})=\int_{\R^{d}} \mathbb{\psi}(|\nabla u_{1}|^{d-1}, |\nabla u_{2}|^{d-1},\ldots,|\nabla u_{n-1}|^{d-1})\,dx
\end{equation}
where ${\bf u}=(u_{1},\ldots,u_{n-1})\in H_{1}\times H_{2} \times \ldots \times H_{n-1}$ is a $2$-tensor. The functional $\mathbb{H}$ is the so-called $k$-harmonic energy, it is modeled on the $(d-1)$-Dirichlet energy. We will consider here a class of energies $\mathbb E$ for maps in $H_{1}\times H_{2} \times \ldots \times H_{n-1}$ which are suitably related to $\mathbb M$ and $\mathbb H$, according to Definition \ref{def:suiten} below. In this case, we investigate the problem of characterizing
$$
\inf \left\{ \mathbb{E}({\bf u}):\,{\bf u}\in H_{1}\times H_{2} \times \ldots \times H_{n-1}  \right\}.
\leqno{(H)}
$$
The main contribution of this paper is the following equivalence result in the minimization problem for the mass $\mathbb M$ and an energy $\mathbb E$ which is suitably related to $\mathbb M$ and $\mathbb H$.

\begin{theorem}\label{thm1}

Assume that a minimizer of the problem $(M)$ admits a calibration (see Definition \ref{Calibration}). Consider an energy functional $\mathbb{E}$ which is suitably related to $\mathbb M$ and $\mathbb H$, in the sense of Definition \ref{def:suiten}. Then, we have
	\begin{equation}\label{thmharmonic}
	\inf{\mathbb{E}}=\alpha_{d-1} \inf{\mathbb{M}}
	\end{equation}
	or equivalently, in view of paper $\cite{MaMa2, MaMa}$,
	\begin{equation}\label{thmharmonic2}
	\inf{\mathbb{E}}=\alpha_{d-1} \inf{I_\alpha}\,.
	\end{equation}
\end{theorem}

Currently, we cannot evade the assumption on the existence of a calibration, because it is still not known if a calibration, or even a weak version of it, is not only sufficient but also a necessary condition for minimality (see Section \ref{section2}). Nonetheless, dropping this assumption we can still state some partial result as follows. 

\begin{remark}{\rm
	\begin{enumerate}
		\item[(i)] If $\alpha=1$, $\psi=\|\cdot \|_{1}$, $\mathbb{E}=\frac{1}{(d-1)^{\frac{d-1}{2}}}\mathbb{H}$, then we are able to prove that $\eqref{thmharmonic}$ still holds true, as a variant of the main result in $\cite{BrezisCoronLieb}$.
		\item[(ii)] In case $0\leq \alpha <1$, we obtain the following inequality
		\begin{equation}\label{compareintro}
		\alpha_{d-1} \inf{\mathbb{M}}=\alpha_{d-1}\inf{I_\alpha}\geq \inf{\mathbb{E}}\, \geq \alpha_{d-1} \inf{\mathbb{N}}\,.
		\end{equation}
		The investigation of equality in $\eqref{compareintro}$ when $0\leq \alpha <1$ is delicate and will be considered in forthcoming works.
	\end{enumerate}
}\end{remark}
\begin{remark}\label{conjecture}
{\rm	We believe that the assumption of the existence of a calibration is not too restrictive. We actually conjecture that minimizing configurations for the problem $(M)$ admit a calibration in case of uniqueness, which is somehow a generic property (see \cite{Cal_Ma_Stein}). We carry out in Example $\ref{examplecalib}$  the construction of configurations of $n$ points in $\R^{n-1}$ with $n-2$ branching points which are generic in character and these configurations admit a calibration.
}\end{remark}
The organization of the paper is as follows: in Section $\ref{section2}$, we briefly review some basic notions of Geometric Measure Theory which will be used in the paper, in Section $\ref{section3}$ we recall (Gilbert-) Steiner problems and briefly describe their connection with Plateau's problem for currents with coefficients in a group. Finally, in Section $\ref{Prooftheorem1}$ we prove the Theorem $\ref{thm1}$.
\section{Preliminaries and notations}\label{section2}
\subsection{Rectifiable currents with coefficients in a group G}
In this section, we present the notion  $1$-dimensional currents with coefficients in the group $\R^{n-1}$ in the ambient space $\R^{d}$ with $n, d\geq 2$. We refer to $\cite{Ma}$ for a more detailed exposition of the subject. 

Consider $\R^{n-1}$ equipped with a norm $\psi$ and its dual norm $\psi^{*}$. Denote by $\Lambda_{1}(\R^{d})$ the space of $1$-dimensional vectors and by $\Lambda^{1}(\R^{d})$ the space of $1$-dimensional covectors in $\R^{d}$.
\begin{definition}{\rm An $(\R^{n-1})^{*}$-valued  $1$-covector on $\R^{d}$ is a bilinear map
	$$w : \Lambda_{1}(\R^{d})\times \R^{n-1}\longrightarrow \R\,.
	$$

	Let $\lbrace e_{1},e_{2},\ldots,e_{n-1} \rbrace$  be an orthonormal basis of $\R^{n-1}$, and let $\lbrace e^{*}_{1},e^{*}_{2},\ldots,e^{*}_{n-1} \rbrace$ be its dual. Then, each $(\R^{n-1})^{*}$-valued  $1$-covector on $\R^{d}$ can be represented as 
	$w=w_{1} e^{*}_{1}+\ldots+w_{n-1}e^{*}_{n-1}\,,$
	where $w_{i}$ is a ``classical'' $1$-dimensional covector in $\R^{d}$ for each $i=1,\ldots,n-1$. To be precise, the action of $w$ on a pair $(\tau,\theta)\in\Lambda_1(\R^d)\times\R^{n-1}$ can be computed as
	\[
\langle w;\tau,\theta\rangle=\sum_{i=1}^{n-1}\theta_i\langle w_i,\tau\rangle\,,	
	\]
	where the scalar product on the right hand side is the standard Euclidean scalar product in $\R^d$.
We denote by $\Lambda^{1}_{(\R^{n-1},\psi)}(\R^{d})$ the space of $(\R^{n-1})^{*}$-valued $1$-covectors in $\R^{d}$, endowed with the (comass) norm:
$$
| w |_{c,\psi}:=\sup \lbrace \psi^{*} ( \langle w ; \tau, \cdot \rangle ) \, : \,  \vert \tau \vert \leq 1\rbrace\,.$$
Similarly, we can define the notion of space $(\R^{n-1})$-valued $1$-vectors in $\R^{d}$, $\Lambda_{1, (\R^{n-1},\psi)}(\R^{d})$, endowed with pre-dual (mass) norm: for any $v\in \Lambda_{1, (\R^{n-1},\psi)}(\R^{d})$ we define:
\begin{equation}\label{nuclearnorm}
\begin{aligned}
| v |_{m,\psi}:= & \sup \lbrace \langle w, v \rangle \, : \,  \vert w \vert_{c,\psi} \leq 1, w\in \Lambda^{1}_{(\R^{n-1},\psi)}(\R^{d}) \rbrace\,\\
= & \inf \left\{ \sum_{l=1}^L \psi (z_l) |\tau_l| \, : \,  \tau_{1},\ldots,\tau_{l} \in \Lambda_{1}(\R^{d}), \, z_1, \ldots, z_k \in \R^{n-1} \mbox{ s.t. }v=\sum_{l=1}^{L}z_l\otimes\tau_{l}  \right\}\,.
\end{aligned}
\end{equation}
}
\end{definition}
\begin{definition}{\rm
An $(\R^{n-1})^{*}$-valued $1$-dimensional differential form defined on $\R^{d}$ is a map 
$$\omega: \R^{d} \longrightarrow \Lambda^{1}_{(\R^{n-1},\psi)}(\R^{d})\,.$$
Let us remark that the regularity of $\omega$ is inherited from the components $\omega_{i}$, $i=1,\ldots,n-1$. 
	Let $\phi=(\phi_1,\ldots,\phi_{n-1})$ be a function of class $C^{1}(\R^d;\R^{n-1})$. We denote
	$${\rm d}\phi:={\rm d\phi_{1}}e^{*}_{1}+\ldots+{\rm d}\phi_{n-1}e^{*}_{n-1},$$
	where ${\rm d}\phi_{i}$ is the differential of $\phi_{i}$. Thus ${\rm d}\phi \in C(\R^{d};\Lambda^{1}_{(\R^{n-1},\psi)}(\R^{d}) )$.
}\end{definition}
\begin{definition}{\rm
A $1$-dimensional current $T$ with coefficients in $(\R^{n-1},\psi)$ is a linear and continuous map
	$$T: C^{\infty}_{c}\left(\R^{d};\Lambda^{1}_{(\R^{n-1},\psi)}(\R^{d})\right) \longrightarrow \R\,.$$
	Here the continuity is meant with respect to the (locally convex) topology on $C^\infty_c(\R^d;\Lambda^1_{(\R^{n-1},\psi)}(\R^d))$ defined in analogy with the topology on $C^\infty_c(\R^d;\R)$ which allows the definition of distributions.
	The mass of $T$ is defined as
	\[
	\mathbb{M}(T):=\sup \left\{ T(\omega):\, \sup_{x\in \R^{d}}|\omega|_{c,\psi} \leq 1 \right\}\,.
	\]
	Moreover, if $T$ is a $1$-dimensional current with coefficients in $(\R^{n-1}, \psi)$, we define the boundary $\partial T$ of $T$ as a distribution with coefficients in $(\R^{n-1},\psi)$, $\partial T: C^{\infty}_{c}(\R^{d};(\R^{n-1},\psi) ) \longrightarrow \R $, such that $$\partial T(\phi):=T({\rm d}\phi)\,.$$ 
	The mass of $\partial T$ is the supremum norm
	\[
	\mathbb{M}(\partial T):=\sup \left\{ T({\rm d}\varphi):\, \sup_{x\in \R^{d}} \psi^*(\varphi)\leq 1 \right\}\,.
	\]
A current $T$ is said to be normal if $\mathbb{M}(T)+\mathbb{M}(\partial T)<\infty$.
}\end{definition}
\begin{definition}{\rm
	A $1$-dimensional rectifiable current with coefficients in the normed (abelian) group $(\Z^{n-1}, \psi)$ is a ($1$-dimensional) normal current (with coefficients in $(\R^{n-1},\psi)$)	such that there exists a $1$-dimensional rectifiable set $\Sigma\subset\R^d$, an approximate tangent vectorfield $\tau \, : \, \Sigma \longrightarrow \Lambda_{1}(\R^{d})$, and a density function $\theta : \Sigma \longrightarrow \Z^{n-1}$ such that 
	$$T(\omega)=\int_{\Sigma}\langle \omega (x) \tau (x), \theta (x) \rangle \,d\mathcal{H}^{1}(x)$$
	for every $\omega \in C^{\infty}_{c}\left(\R^{d};\Lambda^{1}_{(\R^{n-1},\psi)}(\R^{d}) \right)$. We denote such a current $T$ by the triple $\llbracket\Sigma, \tau, \theta\rrbracket$.
}\end{definition}

\begin{remark}\label{rmk:mass}{\rm
The mass of a rectifiable current $T=\llbracket\Sigma,\tau,\theta\rrbracket$ with coefficients in $(\Z^{n-1}, \psi)$  can be computed as
	$$\mathbb{M}(T):=\sup \left\{ T(\omega):\, \sup_{x\in \R^{d}}|\omega|_{c,\psi} \leq 1 \right\}=\int_{\Sigma}\psi (\theta(x))\,d\mathcal{H}^{1}(x)\,.$$
Moreover, $\partial T: C^{\infty}_{c}(\R^{d};(\R^{n-1},\psi) ) \longrightarrow \R $ is a measure and there exist $x_{1},\ldots,x_{m} \in \R^{d}$, $p_{1},\ldots,p_{m} \in \Z^{n-1}$ such that
	$$\partial T(\phi)=\sum_{j=1}^{m}p_{j}\phi(x_{j}).$$
Finally the mass of the boundary $\mathbb{M}(\partial T)$ coincides with $\sum_{j=1}^{m}\psi(p_{j})$.
	}\end{remark}
	\begin{remark}{\rm
		In the trivial case $n=2$, we consider rectifiable currents with coefficients in the discrete group $\Z$ and we recover the classical definition of integral currents (see, for instance, \cite{FeBook}).
	}\end{remark}
Finally, it is useful to define the components $T$ with respect to the index $i\in\{1,\ldots,n-1\}$: for every $1$-dimensional test form $\tilde\omega\in C^\infty_c(\R^d;\Lambda^1(\R^d))$ we set
	$$T^{i}(\tilde\omega):=T(\tilde\omega e^{*}_{i})\,.$$
	Notice that $T^{i}$ is a classical integral current (with coefficients in $\Z$). Roughly speaking, in some situations we are allowed to see a current with coefficients in $\R^{n-1}$ through its components $(T^{1},\ldots,T^{n-1})$.
	
For future convenience we adopt the notation
\begin{equation}\label{infnormalcurrents}
\inf \mathbb{N}:=\min\{\mathbb{M}(T):\,T \text{ is a 1-dimensional normal current with coefficients in }\R^{n-1}\text{ and }\partial T=S\}\,,
\end{equation}
where $S=e_{1}\delta_{P_1}+\ldots+e_{n-1}\delta_{P_{n-1}}-(e_1+e_2+\ldots+e_{n-1})\delta_{P_{n}}$ is a given boundary (and $\lbrace e_i \rbrace_{i=1}^n$ is the canonical basis of $\R^{n-1}$).
	
When dealing with the Plateau problem in the setting of currents, it is important to remark a couple of critical features. For the sake of understandability, we recall them here for the particular case of $1$-dimensional currents, but the matter does not depend on the dimension. 
\begin{remark}\label{lavrentiev}{\rm
If a boundary $\{P_1,\ldots,P_n\}\subset\R^d$ is given, then the problem of the minimization of mass is well posed in the framework of rectifiable currents and in the framework of normal currents as well. In both cases the existence of minimizers is due to a direct method and, in particular, to the closure of both classes of currents. Obviously
\begin{align*}
\inf{\mathbb N}\le & \min\{\mathbb{M}(T):\,T\text{ rectifiable current with coefficients in }\Z^{n-1}\text{ and boundary }\{P_1,\ldots,P_n\}\}\,,
\end{align*}
but whether the inequality is actually an identity is not known for currents with coefficients in groups. The same question about the occurence of a Lavrentiev gap between normal and integral currents holds for classical currents of dimension bigger than $1$ and it is closely related to the problem of the decomposition of a normal current in rectifiable ones (see \cite{Ma} for a proper overview of this issue). }\end{remark}

A formidable tool for proving the minimality of a certain current is to show the existence of a calibration.

\begin{definition}\label{Calibration}{\rm
	Consider a rectifiable current $T=\llbracket\Sigma, \tau, \theta\rrbracket$ with coefficients in $\Z^n$, in the ambient space $\R^{d}$. A smooth $(\R^{n})^{*}$-valued differential form $\omega$ in $\R^{d}$ is a calibration for $T$ if the following conditions hold:
	\begin{enumerate}
		\item[(i)]\label{clr1} for a.e $x\in \Sigma$ we have that $\langle \omega(x); \tau(x), \theta (x)\rangle=\psi (\theta(x));$ 
		\item[(ii)]\label{clr2} the form is closed, i.e, ${\rm d}\omega=0;$
		\item[(iii)]\label{clr3} for every $x\in \R^{d}$, for every unit vector $t \in \R^{d}$ and for every $h\in \Z^{n}$, we have that
		$$\langle \omega(x); t, h \rangle \leq \psi (h)\,.$$
	\end{enumerate}
}\end{definition}

It is straightforward to prove that the existence of a calibration associated to a current implies the minimality of the current itself. Indeed, with the notation in Definition \ref{Calibration}, if $T'=\llbracket\Sigma',\tau',\theta'\rrbracket$ is a competitor, i.e., $T'$ is a rectifiable current with coefficients in $\Z^n$ and $\partial T'=\partial T$, then
\[
{\mathbb M}(T)=\int_{\Sigma}\psi(\theta)=\int_{\Sigma}\langle\omega;\tau,\theta\rangle=\int_{\Sigma'}\langle\omega;\tau',\theta'\rangle\le\int_{\Sigma'}\psi(\theta')={\mathbb M}(T')\,.
\]

We stress that fact that the existence of a calibration is a sufficient condition for the minimality of a current, so it is always a wise attempt when a current is a good candidate for mass minimization. Nonetheless, it is also natural to wonder if every mass minimizing current has its own calibration and this problem can be tackled in two ways: for specific currents or classes of currents (such as holomorphic subvarieties) one has to face an extension problem with the (competing) constraints (ii) and (iii), since condition (i) already prescribes the behaviour of the form on the support of the current. In general, one may attempt to prove the existence of a calibration as a result of a functional argument, picking it in the dual space of normal currents, but this approach has two still unsolved problems:
\begin{itemize}
\item the calibration is merely an element of the dual space of normal currents, thus it is far to be smooth;
\item this argument works in the space of normal currents and it is not known whether a minimizer in this class is rectifiable as well (see Remark \ref{lavrentiev}).
\end{itemize}
Anyway, in this specific case of currents with coefficients in $\Z^n$ which match the energy minimizing networks of a branched optimal transport problem (with a subadditive cost), we think that the Lavrentiev phenomenon cannot occur, as explained in Remark \ref{conjecture}. 

\subsection{Distributional Jacobian}
We recall the definition of distributional Jacobian of a function $u\in W^{1,d-1}_{\rm loc}(\R^{d}; \R^{d})\cap L^{\infty}_{\rm loc}(\R^{d}; \R^{d})$, see also $\cite{JeSo02, ABO1}$.

\begin{definition} Let $u$ be in $W^{1,d-1}_{\rm loc}(\R^{d}; \R^{d})\cap L^{\infty}_{\rm loc}(\R^{d}; \R^{d})$, we define the pre-jacobian $ju \in L^1_{\rm loc}(\R^d;\R^d)$ as
$$ju
:=(\det(u,u_{x_{2}},\ldots,u_{x_{d}}), \det(u_{x_{1}},u,\ldots,u_{x_{d}}),
\ldots,\det(u_{x_{1}},\ldots,u_{x_{d-1}}, u))\,,$$
where $u_{x_j}$ is a $L^{d-1}_{\rm loc}(\R^d;\R^d)$ representative of the partial derivative of $u$ with respect to the $j^{\rm th}$ direction. Thus we define the Jacobian $Ju$ of $u$ as $\frac{1}{d}{\rm d}(ju)$ in the sense of distributions. More explicitly, if $\phi \in C^{\infty}_{c}(\R^{d};\R)$ is a test function, then one has
\begin{equation}\label{distrib_jac}
\int_{\R^{d}}\phi\, Ju\,dx=-\frac{1}{d}\int_{\R^{d}}\nabla \phi \cdot ju\,dx\,.
\end{equation}
The identity required in \eqref{distrib_jac} is clearer if one notices that $ju$ has been chosen in such a way that ${\rm div}(\varphi\tilde u)=\nabla\varphi\cdot j\tilde u+d\varphi\det D\tilde u$ whenever $\tilde u$ is smooth enough to allow the differential computation.
\end{definition}

Once the singularities of the problem ${P_1,\ldots,P_n}$ have been prescribed, we can also introduce the energy spaces $H_{i}$, for each $i=1,\ldots,n-1$. By definition a map $u\in W^{1,d-1}_{\rm loc}(\R^{d}; \mathbb{S}^{d-1})$ belongs to $H_i$ if $Ju=\frac{\alpha_{d-1}}{d}( \delta_{P_i}-\delta_{P_n})$, and there exists a radius $r=r(u)>0$ such that $u$ is constant outside $B(0, r(u))\ni P_{i}, P_{n}$, where
$B(0, r)$ is the open ball of radius $r$ centered at $0$. 

For any $\textbf{u}\in H_1\times \ldots \times H_{n-1}$, we define the (matrix-valued) pre-jacobian of $\textbf{u}$ by
\begin{equation}
\textbf{ju}=(ju_1,\ldots,ju_{n-1})
\end{equation}
and its Jacobian by
\begin{equation}
\textbf{Ju}=(Ju_1,\ldots,Ju_{n-1})\,.
\end{equation}
We observe that $\textbf{ju}$ is actually a $1$-dimensional normal currents with coefficients in $\R^{n-1}$. Moreover
\begin{equation}
\frac{1}{d}\partial \, \textbf{ju}=-Ju\,.
\end{equation}

\begin{definition}\label{def:suiten} Given $P_1,\ldots,P_n\in \R^d$ and a norm $\psi$ on $\R^{n-1}$, a functional $\mathbb{E}$ defined on $H_{1}\times \ldots \times H_{n-1}$ is said to be suitably related to $\mathbb{M}$ and $\mathbb{H}$ (see \eqref{def_h} for its definition) if the following properties hold.
\begin{itemize}
	\item[(i)] $\mathbb{M}(\textbf{\rm\bf ju})\leq \mathbb{E}({\bf u})$, where $\textbf{\rm \bf ju}$ is the normal current defined by the pre-jacobian.
	\item[(ii)] If there exist an open set $U\subset\R^d$ and a subset $I$ of the set of labels ${1,\ldots,n-1}$ such that $u_i=u_l$ for every pair $i,l\in I$ and $u_i=0$ otherwise, we have 
	\begin{equation}
		\mathbb{E}({\bf u}\chi_U) \leq \frac{1}{(d-1)^{\frac{d-1}{2}}} \mathbb{H}({\bf u}\chi_U)\,,
	\end{equation}
	where $\chi_U$ is the characteristic function of $U$.
	\item[(iii)] When $k=1$, the functional $\mathbb{E}$ coincides with the harmonic energy considered in $\cite{BrezisCoronLieb}$.
\end{itemize}
\end{definition}
Let us point out that requirement {\it (ii)} is taylored on the dipole construction maps ${\bf u}=(u_{1},\ldots,u_{n-1})$ in the Step $1$ of the proof of Theorem $\ref{thm1}$.

We consider the following problem:
$$
\inf \left\{\mathbb{E}({\bf u}), \hspace{0.2cm}   {\bf u}=(u_{1},\ldots,u_{n-1})\in H_{1}\times H_{2} \times \ldots \times H_{n-1}  \right\}.
\leqno{(H)}
$$
As indicated in the introduction, the inspiration for considering the problem $(H)$ and comparing it with the irrigation problem $(I)$ is coming from the works $\cite{MaMa2, MaMa}$ and $\cite{abh}$. More precisely, $\cite{MaMa2, MaMa}$ provided a new framework for the problem $(I)$ by proving it to be equivalent to the problem of mass-minimizing currents with coefficients in the group $\Z^{n-1}$ with a suitable norm. The point is to look at each irrigation network $L = \bigcup_{i=1}^{n-1}\lambda_i$ encoded in the current $T=(T^{1},\ldots, T^{n-1})$ where $T^{i}$ is a classical current supported by $\lambda_{i}$, and the irrigation cost of $L$ is the mass of the current $T$. Then, combining this point of view with $\cite{abh}$ (see also $\cite{BrezisCoronLieb}$), where the energy of harmonic maps with prescribed point singularities was related to $1$-dimensional classical currents, we are led to investigate the problem $(H)$ in connection with problem $(I)$. 

Before moving to the next section, we provide a candidate for the functional $\mathbb{E}$ satisfying the properties in Definition \ref{def:suiten}.
Let $\textbf{u}=(u_1,\ldots,u_{n-1}) \in H_1\times \ldots \times H_{n-1}$. Let $e_1,\ldots,e_{n-1}$ be the canonical basis of $\R^{n-1}$, and let $I$ be a subset of $\lbrace 1,\ldots, n-1 \rbrace$, then we denote by
$e_{I}$ the sum $\sum_{i\in I}e_{i}$. We define the energy density $\textbf{e}(\textbf{u})$ at a point $x\in\R^d$ as
\begin{align}
\textbf{e}(\textbf{u})(x)=(d-1)^{-\frac{d-1}{2}}\inf\Bigg\{ & \sum_{I\in{\mathcal{I}}}\|e_I \|_{\alpha}|\nabla u_I(x)|^{d-1}:\, \mbox{where }\textbf{ju}(x)=\sum_{I\in{\mathcal I}}ju_{I}(x)\otimes e_I\,\label{energydensityforE}\\
 & \text{and } \mathcal{I}\text{ is a partition of }\{1,\ldots,n-1\}\Bigg\}\,.\nonumber
\end{align}
To be precise, here the matrix $\textbf{ju} (x)$ is decomposed according to a partition ${\mathcal I}$ of the set $\{1,\ldots,n-1\}$ in such a way that $ju_i(x)=ju_l(x)$ for every pair $i,l\in I$. 

As an example, take ${\bf u}=(u_1,u_2)\in H_1\times H_2$ for some choice of the points $P_1,P_2,P_3\in\R^d$. Then, at some point $x\in \R^d$, either $ju_1(x)\neq ju_2(x)$ or $ju_1(x)=ju_2(x)$.
\begin{itemize}
\item If $ju_1(x)\neq ju_2(x)$, then the unique decomposition that we are allowing is ${\bf j}({\bf u})(x)=ju_1(x)e_1+ju_2(x)e_2$ and $\textbf{e}(\textbf{u})(x)=c_d(|\nabla u_1(x)|^{d-1}+|\nabla u_2(x)|^{d-1})$, where we abbreviated $c_d=(d-1)^{-\frac{d-1}{2}}$.
\item If $ju_1(x)=ju_2(x)$, then, thanks to the subadditivity of $\|\cdot\|_\alpha$, the most convenient decomposition is ${\bf j}({\bf u})(x)=ju_1(x)(e_1+e_2)$ and $\textbf{e}(\textbf{u})(x)=c_d\|e_1+e_2\|_\alpha|\nabla u_1(x)|^{d-1}$.
\end{itemize}
Finally, we consider the functional
\begin{equation}\label{energyforE}
\mathbb{E}(\textbf{u})=\int_{\R^{d}}{\textbf e}(\textbf{u})(x)\, dx.
\end{equation}
\begin{prop}\label{functionalE}Let $\psi$ be the norm 
defined as
\begin{equation}
\psi(h)=\begin{cases} 
||\cdot||_{\alpha}=\left(\sum_{j=1}^{n-1}|h_{j}|^{\frac{1}{\alpha}}\right)^{\alpha} & \mbox{in case } \alpha \in (0; 1], \, h\in \Z^{n-1} \\ 
||\cdot||_{0}=\max \lbrace h_{1},\ldots,h_{n-1} \rbrace & \mbox{in case } \alpha=0, \, h\in \Z^{n-1}\,.
\end{cases}
\end{equation}
Let $\mathbb{E}$ be the functional defined above, in $\eqref{energyforE}$. If $\alpha=1$, i.e. $\psi=\|\cdot \|_{1}$, we choose $\mathbb{E}=\frac{1}{(d-1)^{\frac{d-1}{2}}}\mathbb{H}$. 
Then $\mathbb{E}$ is suitably related to $\mathbb{M}$ and $\mathbb{H}$ in the sense of Definition \ref{def:suiten}.
\end{prop}
\begin{proof}
We start with property {\it (i)}. Let $\omega \in C^{\infty}_{c}\left(\R^{d};\Lambda^{1}_{(\R^{n-1},\psi)}(\R^{d})\right)$ be a test form with comass norm $\sup_{x\in \R^d} |\omega \,|_{c,\psi} \leq 1$. By using the very definition of $|\cdot |_{m,\psi}$, see \eqref{nuclearnorm}, we obtain
\begin{equation}\label{comparenuclearnorm}
\begin{aligned}
|\, \textbf{ju} (\omega)\,|=\left|\int_{\R^d}\langle \textbf{ju}(x), \omega(x) \rangle\, dx\right| \leq \int_{\R^{d}} |\textbf{ju} (x)|_{m,\psi}\,dx\,.
\end{aligned}
\end{equation}
On the other hand, as already observed, for a.e $x\in \R^d$ we have
\begin{equation*}
\begin{aligned}
|\textbf{ju}(x)|_{m,\psi}\leq \inf\left\{\sum_{I\in{\mathcal{I}}}\|e_I \|_{\alpha}|j u_I(x)|:\, \mbox{where }\textbf{ju}(x)=\sum_{I\in{\mathcal I}}ju_{I}(x)\otimes e_I,\,\mathcal{I}\text{ part. of }\{1,\ldots,n-1\}\right\}\,.
\end{aligned}
\end{equation*}
Observe that for any $v\in H_{l}$, $l=1,\ldots,n-1$, one has for a.e $x\in \R^d$
\begin{equation}
|jv(x)|\leq \frac{1}{(d-1)^{\frac{d-1}{2}}}|\nabla v(x)|^{d-1}\,,
\end{equation}
see also $\cite{BrezisCoronLieb}$-page 64, $\cite{abh}$-A.1.3.
Therefore, we obtain that for a.e $x\in \R^d$
\begin{equation}
|\textbf{ju}|_{m,\psi}(x)\leq e(\textbf{u})(x)
\end{equation}
This in turn implies that
\begin{equation}
\begin{aligned}
|\, \textbf{ju} (\omega)\,|\leq \mathbb{E} (\textbf{u})\,.
\end{aligned}
\end{equation}
So, by the arbitrariness of $\omega$, we conclude that
\begin{equation}
\mathbb{M}(\textbf{ju})\leq \mathbb{E} (\textbf{u}).
\end{equation} 

Concerning property {\it(ii)}, assume that, in some open set $U$, each $u_{i}$ is equal to either $0$ or a given function $v\in W^{1,d-1}_{\rm loc}(\R^d,\mathbb{S}^{d-1})$, thus in $U$ the jacobian $\textbf{ju}$ can be written as $\textbf{ju}=jv {e}_{I}$, for some $I\subset \lbrace 1,\ldots,n \rbrace$. This implies that
\begin{equation}
{\bf e}(\textbf{u})(x)\leq \| e_{I} \|_{\alpha} \frac{1}{(d-1)^{\frac{d-1}{2}}} |\nabla v (x)|^{d-1}
\end{equation}
for a.e $x$ in the dipole, so we can conclude that
\begin{equation}
\mathbb{E}(\textbf{u}\chi_U)\leq \mathbb{H}(\textbf{u}\chi_U)\,.
\end{equation}

Finally, if $k=1$ (i.e., we have just one component $\textbf{u}=u$), it is obvious that
\begin{equation}
{\bf e}(\textbf{u})=\frac{1}{(d-1)^{\frac{d-1}{2}}} |\nabla u|^{d-1}.
\end{equation}
To conclude the proof, we observe that, in case $\alpha=1$, that is, $\psi=\| \cdot \|_{1}$, $\mathbb{E}=\frac{1}{(d-1)^{\frac{d-1}{2}}}\mathbb{H}$ and this functional obviously satisfies the three properties.
\end{proof}
\section{(Gilbert-)Steiner problems and currents with coefficients in a group}\label{section3}
Let us briefly recall the Gilbert-Steiner problem and the Steiner tree problem and see how it can be turned into a mass-minimization problem for integral currents in a suitable group. 

Let $n$ distinct points $P_{1},\ldots, P_{n}$ in $\R^{d}$ be given. Denote by $G(A)$ the set of all acyclic graphs $L = \bigcup_{i=1}^{n-1}\lambda_i$, along which the unit masses located at  $P_{1},\ldots, P_{n-1}$ are transported to the target point $P_n$ (single sink). Here $\lambda_i$ is a simple rectifiable curve and represents the path of the mass at $P_{i}$ flowing from $P_{i}$ to $P_{n}$. In $\cite{MaMa2, MaMa}$, the occurrence of cycles in minimizers is ruled out, thus the problem $(I)$ is proved to be equivalent to 

$$
\inf \left\{ \int_L |\theta(x)|^\alpha d{\mathcal H}^1(x), \;\; L\in G(A),  \;\;\theta(x) = \sum_{i=1}^{n-1} \mathbf{1}_{\lambda_i}(x) \right\}
\leqno{(I)}
$$
where $\theta$ is the mass density along the network $L$. Moreover, in $\cite{MaMa2, MaMa}$ the problem $(I)$ can be turned into a mass-minimization problem for integral currents with coefficients in the group $\Z^{n-1}$: the idea is to label differently the masses located at  $P_{1}, P_{2} \ldots, P_{n-1}$ (source points) and to associate the source points $P_{1},\ldots, P_{n-1}$ to the single sink $P_{n}$. Formally, we produce a $0$-dimensional rectifiable current (a.k.a. a measure) with coefficients in $\Z^{n-1}$, given by the difference between
$$\mu^{-}=e_{1}\delta_{P_{1}}+e_{2}\delta_{P_{2}}+ \ldots +e_{n-1}\delta_{P_{n-1}} \mbox{ and }\mu^{+}=(e_{1}+\ldots+e_{n})\delta_{P_{n}}\,.$$
We recall that $\lbrace e_{1},e_{2},\ldots,e_{n} \rbrace$ is the canonical basis of $\R^{n-1}$. The measures $\mu^{-}, \mu^{+}$ are the marginals of the problem $(I)$. To any acyclic graph $L = \bigcup_{i=1}^{n-1}\lambda_i$ we associate a current $T$ with coefficients in the group $\Z^{n-1}$ as follows: to each $\lambda_{i}$ associate the current $T_{i}=\llbracket\lambda_{i},\tau_{i},e_{i}\rrbracket$, where $\tau_{i}$ is the tangent vector of $\lambda_{i}$. We associate to the graph $L = \bigcup_{i=1}^{n-1}\lambda_i$ the current $T=(T_{1},\ldots,T_{n-1})$ with coefficients in $\Z^{n-1}$. By construction we obtain $$\partial T=\mu^{+}-\mu^{-}\,.$$
Choosing the norm $\psi$ on $\Z^{n-1}$ as
\begin{equation}\label{normeuclidean}
\psi(h)=\begin{cases} 
||\cdot||_{\alpha}=\left(\sum_{j=1}^{n-1}|h_{j}|^{\frac{1}{\alpha}}\right)^{\alpha} & \mbox{in case } \alpha \in (0; 1], \, h\in \Z^{n-1} \\ 
||\cdot||_{0}=\max \lbrace h_{1},\ldots,h_{n-1} \rbrace & \mbox{in case } \alpha=0, \, h\in \Z^{n-1}\,,
\end{cases}
\end{equation}


in view of Remark \ref{rmk:mass}, the problem $(I)$ is equivalent to 
$$
\inf \left\{ \mathbb{M}(T), \hspace{0.2cm}  \partial T =\mu^{+}-\mu^{-} \right\}
\,.\leqno{(M)}
$$
We refer the reader to $\cite{MaMa2, MaMa}$ for more details. From now on we restrict our attention to the coefficients group $(\Z^{n-1}, ||\cdot||_{\alpha})$, $0\leq \alpha \leq 1$.
\begin{remark}\label{prejacobiancurrent}
Let $\textbf{u}=(u_1,\ldots,u_{n-1})\in H_1\times \ldots \times H_{n-1}$. One has
\begin{equation}\label{boundaryofprejacobian}
\frac{1}{\alpha_{d-1}} \partial \,\textbf{ju}=\mu^{+}-\mu^{-}
\end{equation}
\end{remark}

We remark that turning the problem $(I)$ into a mass-minimization problem allows to rely on the (dual) notion of calibration, which is a useful tool to prove minimality, especially when dealing with concrete configurations. We also recall that the existence of a calibration (see Definition \ref{Calibration}) associated with a current $T$ implies that $T$ is a mass-minimizing current for the boundary $\partial T$.
\begin{example}\label{examplecalib}{\rm
		Let us consider an irrigation problem with $\alpha=\frac{1}{2}$. We will consider a minimal network joining $n+1$ points in $\R^{n}$, the construction of the network is explained below. Let us stress that in this example the coincidence of the dimension of the ambient space with the dimension of the space of coefficients is needed.
		
Adopting the point of view of \cite{HarveyLawson}, we propose a calibration first, and only {\it a posteriori} we construct a current which fulfills the requirement (i) in Definition \ref{Calibration}. We briefly remind that the problem $(I)$ can be seen as the mass-minimization problem for currents with coefficients in $\Z^{n}$ with the norm $\Vert \cdot \Vert_{\frac{1}{2}}$.  

Let $\{{\rm d}x_1,\ldots,{\rm d}x_n\}$ be the (dual) basis of covectors of $\R^n={\rm span}(e_1,\ldots,e_n)$. We now prove that the differential form 
		\[
		\omega=
		\begin{bmatrix}
		{\rm d}x_{1}\\
		{\rm d}x_{2}\\
		\vdots \\
		{\rm d}x_{n}
		\end{bmatrix}
		\]
		satisfies conditions (ii) and (iii) in Definition $\ref{Calibration}$. Obviously ${\rm d}\omega=0$. Moreover,
		let $\tau=(\tau_{1}, \tau_{2},\ldots,\tau_{n})\in \R^{n}$ be a unit vector (with respect to the Euclidean norm). Thus, for our choice of the norm $\psi=\|\cdot\|_{\frac 12}$ we can compute $\Vert \langle \omega; \tau, \cdot \rangle \Vert^{\frac{1}{2}}=( \tau_{1}^{2}+\tau_{2}^{2}+\tau_{3}^{2}+\ldots+\tau_{n}^{2})^{\frac{1}{2}}=1$. 
		
We will build now a configuration of $n+1$ points $P_{1}, P_{2}, \ldots, P_{n+1}$  in $\R^{n}$ calibrated by $\omega$. Notice that the network has $n-1$ branching points and is somehow generic in character. More precisely, our strategy in building such a configuration is to choose end points, and branching points following the directions parallel to $e_{1}, e_{2}, e_{3}, \ldots, e_{n}, e_{1}+e_{2},e_{1}+e_{2}+e_{3}, \ldots,e_{1}+e_{2}+\ldots+e_{n-1}, e_{1}+e_{2}+\ldots+e_{n}$. We illustrate the construction in $\R^{3}, \R^{4}$. This process can be extended to any dimension.

\begin{itemize}
\item In $\R^{3}$, let us consider $P_{1}=(-1, 0, 0)$, $P_{2}=(0, -1, 0)$, $P_{3}=(1, 1, -1)$, $P_{4}=(2,2 ,1)$. Take, as 
			branching points, $G_{1}=(0, 0, 0)$,  $G_{2}=(1, 1, 0)$. Now consider the current $T=\llbracket\Sigma, \tau, \theta\rrbracket$ with support $\Sigma$ obtained by the union of the segments $\overline{P_1G_1},\overline{P_2G_1},\overline{G_1G_2},\overline{P_3G_2},\overline{G_2P_4}$. 
			\begin{figure}[tbh]
				\centering
				\begin{tabular}{cc}
					\includegraphics[width=0.4\linewidth]{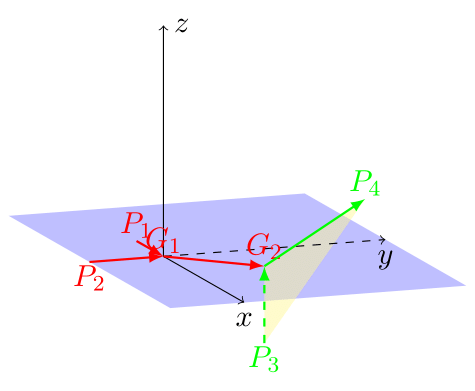}
				\end{tabular}
				\caption{The picture illustrates the construction of $T$.}
				\label{fig:1d_exe}
			\end{figure}
			
			The multiplicity $\theta$ is set as
			$$\theta(x)
			=\begin{cases} 
			e_{1} & \mbox{if } x\in \overline{P_{1}G_{1}} \\ 
			e_{2} & \mbox{if } x\in \overline{P_{2}G_{1}} \\
			e_{1}+e_{2} & \mbox{if } x\in \overline{G_{1}G_{2}} \\
			e_{3} & \mbox{if } x\in \overline{P_{3}G_{2}} \\
			e_{1}+e_{2}+e_{3} & \mbox{if } x\in \overline{G_{2}P_{4}} \\
			0 & \mbox{elsewhere}.\\
			\end{cases}$$			
We observe that $T$ is calibrated by $\omega$, thus $T$ is a minimal network for the irrigation problem with sources $P_1,P_2$ and $P_3$ and sink $P_4$. Notice that edges of the network meet at the branching points with the $90$ degrees angles, as known for branched optimal structures with cost determined by $\alpha=1/2$.
\item In $\R^{4}$, we keep points $P_{1}=(-1, 0, 0, 0)$, $P_{2}=(0, -1, 0, 0)$, $P_{3}=(1, 1, -1, 0)$ and, in general, the whole network of the example above as  embedded in $\R^{4}$. We relabel $G_{3}:=(2,2,1,0)$. We now pick $P_{4}$ and $P_{5}$ in such a way that $\overrightarrow{P_4G_{3}}=e_{4}$ and $\overrightarrow{G_{3}P_{5}}=e_{1}+e_{2}+e_{3}+e_{4}$. For instance, we choose	$P_{4}=(2, 2, 1, -1)$ and $P_{5}=(3, 3, 2, 1)$. As before, the marginals of the irrigation problem are $P_1,P_2,P_3,P_4$ as sources and $P_5$ as sink, while $G_1,G_2,G_3$ are branching points.

Let us now consider the current $T=\llbracket\Sigma, \tau, \theta\rrbracket$ supported on the union of segments $\overline{P_1,G_1},\overline{P_2G_1},\overline{G_1G_2},\overline{P_3G_2},\overline{G_2G_3},\overline{P_4G_3},\overline{G_3P_5}$ and multiplicity $\theta$ given by
$$
			\theta(x)
			=\begin{cases} 
			e_{1} & \mbox{if } x\in \overline{P_{1}G_{1}} \\ 
			e_{2} & \mbox{if } x\in \overline{P_{2}G_{1}} \\
			e_{1}+e_{2} & \mbox{if } x\in \overline{G_{1}G_{2}} \\
			e_{3} & \mbox{if } x\in \overline{P_{3}G_{2}} \\
			e_{1}+e_{2}+e_{3} & \mbox{if } x\in \overline{G_{2}G_{3}} \\
			e_{4} & \mbox{if } x\in \overline{P_{4}G_{3}} \\
			e_{1}+e_{2}+e_{3}+e_{4} & \mbox{if } x\in \overline{G_{3}P_{5}} \\
			0 & \mbox{elsewhere}.\\
			\end{cases}$$
			It is easy to check that the orientation of each segment coincides with the multiplicity, 
			therefore $T$ is calibrated by $\omega$.		
			\item This procedure can be replicated to construct a configuration of $n+1$ points $P_{1}, P_{2}, \ldots, P_{n+1}$ in $\R^{n}$ calibrated by $\omega$, always in the case $\alpha=1/2$.	
		\end{itemize}
		}\end{example}
		\begin{example}{\rm
			We now consider a Steiner tree problem. As in the previous example, we aim to construct calibrated configurations joining $n+1$ points in $\R^{n}$ (with $n-1$ branching points). Consider the following differential form:
			\[
			\omega=
			\begin{bmatrix}
			\frac{1}{2}{\rm d}x_{1}+\frac{\sqrt{3}}{2}{\rm d}x_{2}\\
			\frac{1}{2}{\rm d}x_{1}-\frac{\sqrt{3}}{2}{\rm d}x_{2}\\
			\frac{-1}{2}{\rm d}x_{1}-\frac{\sqrt{3}}{2}{\rm d}x_{3}\\	
			\frac{-1}{4}{\rm d}x_{1}+\frac{\sqrt{3}}{4}{\rm d}x_{3}-\frac{\sqrt{3}}{2}{\rm d}x_{4}\\
			\frac{-1}{8}{\rm d}x_{1}+\frac{\sqrt{3}}{8}{\rm d}x_{3}+\frac{\sqrt{3}}{4}{\rm d}x_{4}-\frac{\sqrt{3}}{2}{\rm d}x_{5}\\
			\vdots\\
			\frac{-1}{2^{n-2}}{\rm d}x_{1}+\frac{\sqrt{3}}{2^{n-2}}{\rm d}x_{3}+\frac{\sqrt{3}}{2^{n-3}}{\rm d}x_{4}+\ldots+\frac{\sqrt{3}}{2^{n-k}}{\rm d}x_{k+1}+\ldots+\frac{\sqrt{3}}{4}{\rm d}x_{n-1}-\frac{\sqrt{3}}{2}{\rm d}x_{n}
			\end{bmatrix}\,.
			\]
			It is easy to check that the differential form $\omega$ is a calibration only among those currents having multiplicities 
			$e_{1}, e_{2}, e_{3}, \ldots, e_{n}, e_{1}+e_{2},e_{1}+e_{2}+e_{3}, \ldots,e_{1}+e_{2}+\ldots+e_{n-1}, e_{1}+e_{2}+\ldots+e_{n}$ and hence it will allow to prove the minimality of configurations in the class of currents with those multiplicities  (cf.\cite{Marcello} for the notion calibrations in families). Nevertheless, it is enough to prove the minimality of global minimizers in some configurations.
			
		\begin{itemize}
			\item 	 Consider $n=3$ and
			$P_{1}=\left(\frac{-1}{2},\frac{\sqrt{3}}{2},0\right)$, $P_{2}=\left(\frac{-1}{2},\frac{-\sqrt{3}}{2}, 0\right)$, $P_{3}=\left(\frac{\sqrt{6}}{2}-\frac{1}{2}, 0 ,\frac{\sqrt{3}}{2}\right)$, $P_{4}=\left(\frac{\sqrt{6}}{2}-\frac{1}{2}, 0 , -\frac{\sqrt{3}}{2}\right)$ (see also the example in \cite[Section $3$]{BoOrOu2}).
Indeed, we observe that the lengths $|\overline{P_{1}P_{2}}|=|\overline{P_{1}P_{3}}|=|\overline{P_{1}P_{4}}|=|\overline{P_{2}P_{3}}|=|\overline{P_{2}P_{4}}|=|\overline{P_{3}P_{4}}|=\sqrt{3}$, meaning that the convex envelope of points $P_{1},P_{2},P_{3},P_{4}$ is a tetrahedron: this observation allows us to restrict our investigation among all currents having multiplicities $e_{1}, e_{2}, e_{3}, e_{1}+e_{2}, e_{1}+e_{2}+e_{3}$. More precisely,  given any $1$-dimensional integral current $T$ with $\partial T=(e_{1}+e_{2}+e_{3})\delta_{P_{4}}-e_{1}\delta_{P_{1}}-e_{2}\delta_{P_{2}}-\ldots -e_{3}\delta_{P_{3}}$ whose support is an acyclic graph with two additional Steiner points, we can always construct a corresponding current $L$ with multiplicities $e_{1}, e_{2}$, $e_{1}+e_{2}$, $e_{1}+e_{2}+e_{3}$ having the same boundary with $T$ such that $\mathbb{M}(T)=\mathbb{M}(L)$ thanks to the symmetric configuration $P_{1}, P_{2}, P_{3}, P_{4}$ combined with the fact that any minimal configuration cannot have less than two Steiner points. Indeed, by contradiction, if a minimal configuration for the vertices of a tetrahedron had $1$ Steiner point, then this configuration would violate the well-known property of the $120$ degrees angles at Steiner points.
Therefore, $\omega$ calibrates the current $T=\llbracket\Sigma, \tau, \theta\rrbracket$, where $S_{1}=(0, 0, 0), S_{2}=\left(\frac{\sqrt{6}}{2}-1, 0, 0\right)$ are the Steiner points, $\Sigma=\overline{P_{1}S_{1}}\cup \overline{P_{2}S_{1}} \cup \overline{S_{1}S_{2}} \cup \overline{P_{3}S_{2}} \cup \overline{S_{2}P_{4}}$ and the multiplicity is given by
			$$\theta(x)
			=\begin{cases} 
			e_{1} & \mbox{if } x\in \overline{P_{1}S_{1}} \\ 
			e_{2} & \mbox{if } x\in \overline{P_{2}S_{1}} \\
			e_{1}+e_{2} & \mbox{if } x\in \overline{S_{1}S_{2}} \\
			e_{3} & \mbox{if } x\in \overline{P_{3}S_{2}} \\
			e_{1}+e_{2}+e_{3} & \mbox{if } x\in \overline{S_{2}P_{4}} \\
			0 & \mbox{elsewhere}\,.\\
			\end{cases}$$	
			\item Using the same strategy of Example \ref{examplecalib}, we can build a configuration $P_{1}, P_{2}, P_{3}, P_{4}, P_{5}$ in $\R^{4}$ starting from the points $P_{1}, P_{2}, P_{3}, P_{4}$ above, in such a way that the new configuration is calibrated by $\omega$ among all currents with multiplicities $e_{1}, e_{2}, e_{3}, e_{4}, e_{1}+e_{2},e_{1}+e_{2}+e_{3}, e_{1}+e_{2}+e_{3}+e_{4}$. This construction can be extended to any dimension.
		\end{itemize}
}\end{example}

\section{Proof of the main result}\label{Prooftheorem1}

The proof of Theorem $\ref{thm1}$ is much in the spirit of the dipole construction of $\cite{BrezisCoronLieb, abh}$ (in the version of $\cite{ABO1}$), the properties of the functional $\mathbb{E}$, and making use of the existence of calibration. 
\begin{proof}
Let $\mathbb{E}$ be the functional which fulfills the requirements of Definition \ref{def:suiten}.
	In the first steps we prove the inequality 
	$$\inf{\mathbb{E}}\leq \alpha_{d-1} \inf{I_\alpha}.$$
	
We briefly recall the dipole construction (see, for instance, \cite[Theorem $3.1$, Theorem $8.1$]{BrezisCoronLieb}). Given a segment $\overline{AB}\subset\R^d$ and a pair of parameters $\beta,\gamma>0$, we define 
\begin{equation}\label{neigh}
U:=\{x\in\R^d:\,{\rm dist}(x,\overline{AB})<\min\{\beta,\gamma\,{\rm dist}(x,\{A,B\})\}\}\subset\R^d
\end{equation} 
to be a pencil-shaped neighbourhood with core $\overline{AB}$ and parameters $\beta,\gamma$. For any fixed $\varepsilon>0$, the dipole construction produces a function $u\in W^{1,d-1}_{\rm loc}(\R^{d}; \mathbb{S}^{d-1})$ with the following properties:
	\begin{itemize}
	\item $u\equiv (0,\ldots,0,1)$ in $\R^d\setminus U$;
	\item $Ju=\frac{\alpha_{d-1}}{d}(\delta_{A}-\delta_{B})$;
	\item moreover the map $u$ satisfies the following inequality
	\begin{equation}\label{estimateharmonic}
	\frac{1}{(d-1)^{\frac{d-1}{2}}\alpha_{d-1}}\int_{\R^{d}}|\nabla u|^{d-1}dx \leq |AB|+\varepsilon\,,
	\end{equation}
	\end{itemize}	
\textbf{Step 1.} 	
Let $L=\bigcup_{i=1}^{n-1}\lambda_i$ be an acyclic connected polyhedral graph, and $T$ be the associated current with coefficients in $\Z^{n-1}$ corresponding to $L$. Since $L$ is polyhedral, it can also be written as $L=\bigcup_{j=1}^{k} I_{j}$, where $I_j$ are weighted segments. For each segment $I_{j}$ we can find parameters $\delta_j,\gamma_j>0$ such that the pencil-shaped neighbourhood $U_j = \left\{ x \in \R^d:\, \text{dist}(x,I_{j}) \leq \min \left\{ \beta_{j}, \gamma_j\text{dist}(x,\partial I_{j}) \right\} \right\}$ (modelled after \eqref{neigh}) is essentially disjoint from $U_\ell$ for every $\ell\neq j$. Then, for every $i=1,\ldots,n-1$, let 
	$V_{i}=\bigcup_{j\in K_i} U_j$
	be a sharp covering of the path $\lambda_{i}$. To be precise, we choose $K_i\subset\{1,\ldots,k\}$ such that $V_i\cap U_\ell$ is at most an endpoint of the segment $I_\ell$, if $\ell\notin K_i$.
		\begin{figure}[tbh]
		\centering
		\begin{tabular}{cc}
			\includegraphics[width=0.4\linewidth]{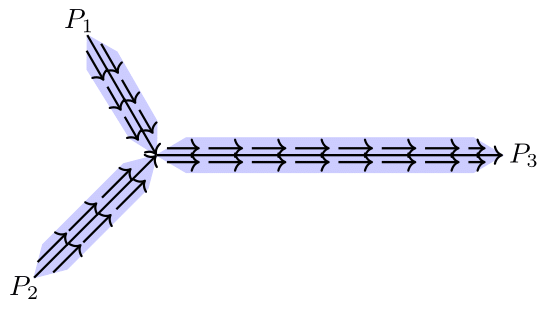}
		\end{tabular}
		\caption{A dipole construction of a Y-shaped graph connecting $3$ points.}
		\label{fig:1d_exe2}
		\end{figure}

For each path $\lambda_i$, $i=1,\ldots,n-1$, we build the map $u_{i}\in H_{i}$ in such a way that it coincides with a dipole associated to the segment $I_j$ in the neighbourhood $U_j$ for each $j\in K_i$. We put $u_i\equiv (0,\ldots,0,1)$ in $\R^d\setminus V_i$.
	
We obtain that $u_{i}\in W^{1,d-1}_{\rm loc}(\R^{d}; \mathbb{S}^{d-1})$ and satisfies $Ju_{i}=\frac{\alpha_{d-1}}{d}(\delta_{P_{i}}-\delta_{P_{n}})$. Moreover, summing up inequality \eqref{estimateharmonic} repeated for each segment $I_j$ with $j\in K_i$, the following inequality holds
	\begin{equation*}
	\frac{1}{(d-1)^{\frac{d-1}{2}}\alpha_{d-1}}\int_{\R^{d}}|\nabla u_{i}|^{d-1}dx \leq \mathbb{M}(T_{i})+k\varepsilon\,,
	\end{equation*}
	where $T_{i}$ is the (classical) integral current corresponding to the $i^{\rm th}$ component of $T$.

In particular, let us stress that the maps $u_1,\ldots,u_{n-1}$ have the following further property: if some paths $\lambda_{i_1},\lambda_{i_2},\ldots,\lambda_{i_m}$ have a common segment $I_j$ for some $j\in K_{i_1}\cap K_{i_2}\cap\ldots\cap K_{i_m}$, then $u_{i_1},\ldots,u_{i_m}$ agree in $U_j$. Furthermore, setting $h_{i_{1}, i_{2},\ldots,i_{m}}=(0,\ldots,|\nabla u_{i_{1}}|^{d-1},\ldots, |\nabla u_{i_{m}}|^{d-1},\ldots,0)$, we obtain
\begin{equation*}
		\frac{1}{(d-1)^{\frac{d-1}{2}}\alpha_{d-1}}\int_{U_j}||h_{i_{1}, i_{2},\ldots,i_{m}}||_{\alpha}dx \leq m^{\alpha}(|I_j|+k\varepsilon)\,,
		\end{equation*}
		where $h_{i_{1}, i_{2},\ldots,i_{m}}=(0,\ldots,|\nabla u_{i_{1}}|^{d-1},\ldots, |\nabla u_{i_{m}}|^{d-1},\ldots,0)$. This holds for every $\alpha\in[0,1]$.
	
	Combining all the previous observations, we can conclude that, given any $\tilde\epsilon >0$ , there exist $u_{i}\in H_{i}$, $i=1,\ldots,n-1$ such that 
	\begin{align*}
	\int_{\R^{d}} ||(|\nabla u_{1}|^{d-1}, |\nabla u_{2}|^{d-1},\ldots,|\nabla u_{n-1}|^{d-1})||_{\alpha}\,dx
	\leq & (d-1)^{\frac{d-1}{2}}\alpha_{d-1} \int_{L}|\theta (x)|^{\alpha}d\mathcal{H}^{1}(x)+\tilde\epsilon
	\\ = & (d-1)^{\frac{d-1}{2}}\alpha_{d-1}\mathbb{M}(T)+\tilde\epsilon\,,
	\end{align*}
	where $\theta(x) = \sum_{i=1}^{n-1} \mathbf{1}_{\lambda_i}(x)$.
Thus, by the properties of $\mathbb{E}$, one obtain that
\begin{equation}
\inf \mathbb{E} \leq \mathbb{E}(\textbf{u})\leq  \frac{1}{(d-1)^{\frac{d-1}{2}}} \mathbb{H}(\textbf{u}) \leq \alpha_{d-1}\mathbb{M}(T)+\tilde\epsilon.
\end{equation}
	
	\noindent {\bf Step 2.} Considering an arbitrary acyclic graph $L=\bigcup_{i=1}^{n-1}\lambda_i$, there is a sequence of acyclic polyhedral graphs $\left(L_{m} \right)_{m\ge 1}$, $L_{m}=\bigcup_{i=1}^{n-1}\lambda^{m}_i$ such that
	the Hausdorff distance $d_{H}(\lambda^m_i, \lambda_i) \leq \frac{1}{m}$, moreover (see \cite[Lemma $3.10$]{BoOrOu}) denoting by $T$ and $T_{m}$ the associated currents with coefficients in $\Z^{n-1}$ we also have that
	$$\mathbb{M}(T_{m})=\int_{L_{m}}|\theta_{m}(x)|^{\alpha}\,d\mathcal{H}^{1}(x) \leq \mathbb{M}(T)=\int_{L}|\theta(x)|^{\alpha}\,d\mathcal{H}^{1}(x) +\frac{1}{m}.$$
	here  $\theta_{m}(x) = \sum_{i=1}^{n-1} \mathbf{1}_{\lambda_i^{m}}(x)$. 
	On the other hand, by previous construction there exists a sequence $\lbrace {\bf u}_{m} \rbrace_{m}$, ${\bf u}_{m}=(u_{1,m},\ldots,u_{n-1,m})\in H_{1}\times \ldots \times H_{n-1}$ such that
	\begin{align*}
	\inf \mathbb{E} \leq \mathbb{E}(\textbf{u}_{m})\leq \frac{1}{(d-1)^{\frac{d-1}{2}}}  \mathbb{H}(\textbf{u}_{m})
	&\leq \alpha_{d-1} \int_{L_{m}}|\theta_{m}(x)|^{\alpha}d\mathcal{H}^{1}(x)+\frac{1}{m}\\
	&=\alpha_{d-1}\mathbb{M}(T_{m})+\frac{1}{m}\\
	&\leq \alpha_{d-1}\mathbb{M}(T)+\frac{1+\alpha_{d-1}}{m}\\
	&=\alpha_{d-1} \int_{L}|\theta(x)|^{\alpha}d\mathcal{H}^{1}(x)+\frac{1+\alpha_{d-1}}{m}.\,,\\
	\end{align*}This implies that
	\begin{equation}
	\inf{\mathbb{E}}\leq \alpha_{d-1} \inf{I_\alpha}=\alpha_{d-1} \inf \mathbb{M}.
	\end{equation}
On the other hand, by the properties $(i)$ of Definition \ref{def:suiten}, we also have that for any $\textbf{u}=(u_1,\ldots,u_{n-1})\in H_{1}\times \ldots \times H_{n-1}$
\begin{equation}
\alpha_{d-1} \inf \mathbb{N} \leq  \mathbb{M}(\textbf{ju})  \leq \mathbb{E}(\textbf{u})
\end{equation}
(see Remark \ref{prejacobiancurrent} to see why the constant $\alpha_{d-1}$ appears in front of $\inf \mathbb{N}$ and also see $\eqref{infnormalcurrents}$ for the definition of $\inf \mathbb{N}$).
This allows us to conclude that
\begin{equation}
\alpha_{d-1} \inf \mathbb{N} \leq  \inf \mathbb{E}.
\end{equation}
Therefore we obtain the following inequality:
\begin{equation}
\alpha_{d-1} \inf \mathbb{N} \leq \inf{\mathbb{E}}\leq \alpha_{d-1} \inf{I_\alpha}=\alpha_{d-1} \inf \mathbb{M}.
\end{equation}
By assumption, a minimizer of the problem $(M)$ admits a calibration, we have 
\begin{equation}
\inf \mathbb{N}=\inf \mathbb{M}=\inf{I_\alpha}.
\end{equation}
this also means that 
\begin{equation}
\alpha_{d-1}\inf \mathbb{N}=\alpha_{d-1}\inf \mathbb{M}=\alpha_{d-1}\inf{I_\alpha}=\inf \mathbb{E}
\end{equation}
which is sought the conclusion.
\end{proof}
\begin{remark}{\rm
	In the proof of Theorem $\ref{thm1}$, step 3, we must assume the existence of a calibration  $\omega$.  Observe that, without this assumption, we still can deduce from that
	\begin{equation}\label{compare6}
	\alpha_{d-1} \inf{\mathbb{M}}=\alpha_{d-1}\inf{I_\alpha} \geq \inf{\mathbb{E}} \geq \alpha_{d-1}\, \inf{\mathbb{N}}
	\end{equation}
	where $\inf{\mathbb{N}}$ is the infimum of the problem obtained measuring the mass among $1$-dimensional normal currents with coefficients in $\R^{n-1}$ (see \eqref{infnormalcurrents}). 
	
	Moreover, in case $\alpha=1$, $\psi=\| \cdot \|_{1}$, $\mathbb{E}=\mathbb{H}$. First, $(I)$ turns out to coincide with the Monge-Kantorovich problem. Then,
	$$\inf{\mathbb{H}}\geq (d-1)^{\frac{d-1}{2}}\alpha_{d-1} \inf{I_\alpha}=(d-1)^{\frac{d-1}{2}}\alpha_{d-1} \inf{\mathbb{M}\,.}$$
	To see this is to use the results of Brezis-Coron-Lieb $\cite{BrezisCoronLieb}$ separately for each map $u_{i}$, $i=1,\ldots,n-1$, for the energy
	$$\mathbb{H}({\bf u})=\int_{\R^{d}}(|\nabla u_{1}|^{d-1}+ |\nabla u_{2}|^{d-1}+\ldots+|\nabla u_{n-1}|^{d-1})\,dx\,,$$
	where, again, ${\bf u}=(u_{1},\ldots,u_{n-1})\in H_{1}\times \ldots \times H_{n-1}$.		
	The investigation of equality cases in $\eqref{compare6}$, when $0\leq \alpha <1$, will be considered in forthcoming works.
}\end{remark}

\section*{Acknowledgements}
The authors are partially supported by GNAMPA-INdAM. The research of the third
author has been supported by European Union's Horizon 2020 programme through project 752018 and by STARS@unipd project ``QuASAR - Questions About Structure And Regularity of currents'' (MASS\_STARS\_MUR22\_01).

The authors wish to warmly thank Giacomo Canevari for extremely fruitful and enlightening discussions.

\end{document}